\documentclass[12pt]{amsart}
\usepackage{amsmath,amssymb,mathabx,amsthm,mathtools}
\usepackage{mathabx,overpic}
\usepackage[margin=1.5in]{geometry}
\newtheorem*{thm*}{Theorem}
\newtheorem{thm}{Theorem}[section]
\newtheorem{prop}[thm]{Proposition}
\newtheorem{lem}[thm]{Lemma}
\newtheorem{cor}[thm]{Corollary}
\newtheorem{defi}[thm]{Definition}
\numberwithin{equation}{section}
\def\:{\colon}
\def\C{\mathbb{C}}
\def\R{\mathbb{R}}

\def\N{\mathbb{N}}
\def\Z{\mathbb{Z}}
\def\Q{\mathbb{Q}}
\def\H{\mathbb{H}}
\def\bC{\widehat{\mathbb{C}}}

\def\ga{\gamma} 
\def\Ga{\Gamma} 
\def\om{\omega} 
\def\Om{\Omega} 
\def\la{\lambda} 
\def\ra{\rightarrow} 
\def\sub{\subseteq} 
\newcommand{\im}{\operatorname{Im}}
\newcommand{\re}{\operatorname{Re}}
\newcommand{\inte}{\operatorname{inte}}

\newcommand{\PSL}{\operatorname{PSL}}
\begin{document}
\title{The quasi-periods of the Weierstrass zeta-function}  
\author{Mario Bonk}
\thanks{The author was partially supported by  NSF grant DMS-2054987.}
\address {Department of Mathematics, University of California, Los Angeles,
CA 90095}
\email{mbonk@math.ucla.edu}
\date{November 26, 2024}
\begin{abstract} We study the ratio $p=\eta_1/\eta_2$ of the  
quasi-periods of 
the Weierstrass $\zeta$-function in dependence of the ratio 
$\tau=\om_1/\om_2$ of the gene\-rators of the underlying rank-2 lattice. We will give an explicit geometric description of the map 
$\tau\mapsto p(\tau)$. As a consequence, we obtain an  explanation of a theorem by Heins who showed that $p$ attains every value in the Riemann sphere infinitely often. Our main result is implicit in the classical literature, but it seems not to be very well known. 

Essentially, this is an expository paper.  We hope that it   is easily accessible and may serve  as an introduction to these  classical themes.  \end{abstract}

\maketitle

\section{Introduction}
Throughout this paper, we assume that $\om_1,\om_2\ne 0$ are two numbers in the complex plane $\C$ that are linearly independent over the field of real numbers $\R$.  We  define 
\begin{equation}\label{tau}
\tau\coloneqq \om_1/\om_2, 
\end{equation}
and  assume that 
\begin{equation}
\im(\tau)>0.
\end{equation}
We also consider the associated rank-$2$ lattice 
 \begin{equation}\label{latt} 
\Gamma=\{k\om_1+n\om_2: n,k\in \Z\}.
\end{equation}
generated by $\om_1$ and $\om_2$.
 
The {\em Weierstrass $\zeta$-function} associated with this lattice $\Gamma$  is  given by the series
\begin{equation}\label{zeta}
\zeta(u;\Ga)=\frac 1u+\sum_{\ga\in \Gamma\setminus\{0\}}
\biggl( \frac{1}{u-\ga}+\frac1\ga +\frac{u}{\ga^2}\biggr). 
\end{equation}
We  simply write $\zeta(u)$ if the underlying lattice $\Gamma$ is understood. It is well known that the series  \eqref{zeta} converges absolutely and locally uniformly  for $u\in \C\setminus \Gamma$.  Moreover, $\zeta$ is an odd meromorphic function with poles of first order precisely at the lattice points. It  has the periodicity property
\begin{equation}\label{defeta} 
\zeta(u+\om_k)=\zeta(u)+\eta_k\end{equation}
for some constants $\eta_k\in \C$, $k=1,2$ (for a discussion of all these facts, see  \cite[Chapter~4]{Cha}). Throughout this paper, we assume that $u$ represents a variable  in 
$\C$ and so \eqref{defeta} and similar formulas are identities valid for all $u\in \C$. 

We call $\eta_1$ and $\eta_2$ the {\em quasi-periods} of $\zeta$ associated with the given generators $\om_1$ and $\om_2$ of $\Gamma$.  Evaluating \eqref{defeta}  at $u=-\om_k/2$ and using that $\zeta$ is an odd function, we see that 
\begin{equation}
\eta_k= 2 \zeta(\tfrac12 \om_k)
\end{equation}
for $k=1,2$. 
It also follows from \eqref{defeta}  that  \begin{equation}
\zeta(u+k\om_1+n\om_2)=\zeta(u)+ k\eta_1+n\eta_2
 \end{equation}
for $k,n\in \Z$.

Since there are no doubly-periodic meromorphic functions on $\C$ with period lattice $\Gamma$ and only first-order poles at the lattice points (see \cite[Corollary, p.~23]{Cha}), $\eta_1$ and $\eta_2$ cannot both vanish. This means that 
\begin{equation}
p\coloneqq \eta_1/\eta_2\in \bC\coloneqq\C\cup\{\infty\}
\end{equation}
is well defined.  Moreover, it  follows from the formulas above that  $\eta_1$ and $\eta_2$ are homogeneous  of degree $-1$ when considered as functions of the pair $(\om_1,\om_2)$. This means that if we make the substitution
\begin{equation}
(\om_1,\om_2)\mapsto (t \om_1,t\om_2)
\end{equation}
with $t\in\C^*\coloneqq \C\setminus \{0\}$, then we get a corresponding substitution 
\begin{equation}
(\eta_1,\eta_2)\mapsto (t^{-1} \eta_1,t^{-1}\eta_2).
\end{equation}
This implies that the function $p$ is homogeneous of degree $0$ and hence we can consider  $p=p(\tau)$ as function of 
$\tau=\om_1/\om_2$. 

Let  $\H\coloneqq\{\tau\in \C: \im(\tau)>0\}$ be  the open upper halfplane. The function $\tau\in \H\mapsto p(\tau)\in \bC$ can be described explicitly, as the  
main result of this paper shows.  In order to state this, we first introduce some terminology.

A {\em closed Jordan region} $X$ in the Riemann sphere $\bC$ is a compact set homeo\-morphic to a closed disk. Then its boundary $\partial X$ is a Jordan curve 
and the set of interior points $\inte(X)\coloneqq X\setminus \partial X$ a simply connected open region. 

A {\em circular arc triangle} $T$ is a closed Jordan region in $\bC$ whose boundary is decomposed into  three non-overlapping circular arcs considered as the {\em sides} of $T$. The three endpoints of these arcs are the {\em vertices} of $T$.

We say that $f$ is a {\em conformal map} between two circular arc triangles $X$ and $Y$ in  $\bC$ if $f$ is a homeomorphism between $X$ and $Y$ that sends the vertices of $X$ to the vertices of $Y$ and is a biholomorphism between $\inte(X)$ and  $\inte(Y)$. 
We say that $f$ is an {\em anti-conformal map} between 
$X$ and $Y$ if $z\in X\mapsto \overline{f(z)}\in \overline Y$ is a conformal map between $X$ and the  complex conjugate  $\overline Y\coloneqq\{\overline w: w\in Y\}$ of $Y$.

We can now state our result.

\begin{thm}\label{main}
The function $\tau \mapsto p(\tau)$ is meromorphic  on 
$\H$. It maps the circular arc triangle 
\begin{equation}\label{T0}
T_0\coloneqq \{ \tau\in \C:  0\le \re(\tau) \le 1/2,\, \im(\tau)>0, \, 
  |\tau|\ge 1\} \cup\{\infty\}
\end{equation}
conformally onto the circular arc triangle
\begin{equation} \label{T1} 
T_1\coloneqq \{ \tau\in\C: 0\le \re(\tau) \le 1/2\}\setminus 
\{ \tau\in\C: \im(\tau)<0,\, |\tau|> 1\} \cup\{\infty\}. 
\end{equation}
Here the vertices of the triangles  correspond under the map $p$
in the following way: 
$$p(i)=-i,\,  p(\tfrac12 (1+i\sqrt3))=\tfrac12 (1-i\sqrt3),\,  
p(\infty)=\infty. $$ 
\end{thm}
Of course,  the last relation has to be understood in a limiting sense: if $\tau \in T_0\cap \H\to \infty$, then $p(\tau)\in T_1\cap \H \to \infty$. In the  following, we always think 
of $p$ (and similar functions) as extended to the point $\infty$ 
in this way.  

\begin{figure}[b]
 \begin{overpic}[ scale=0.5
    ]{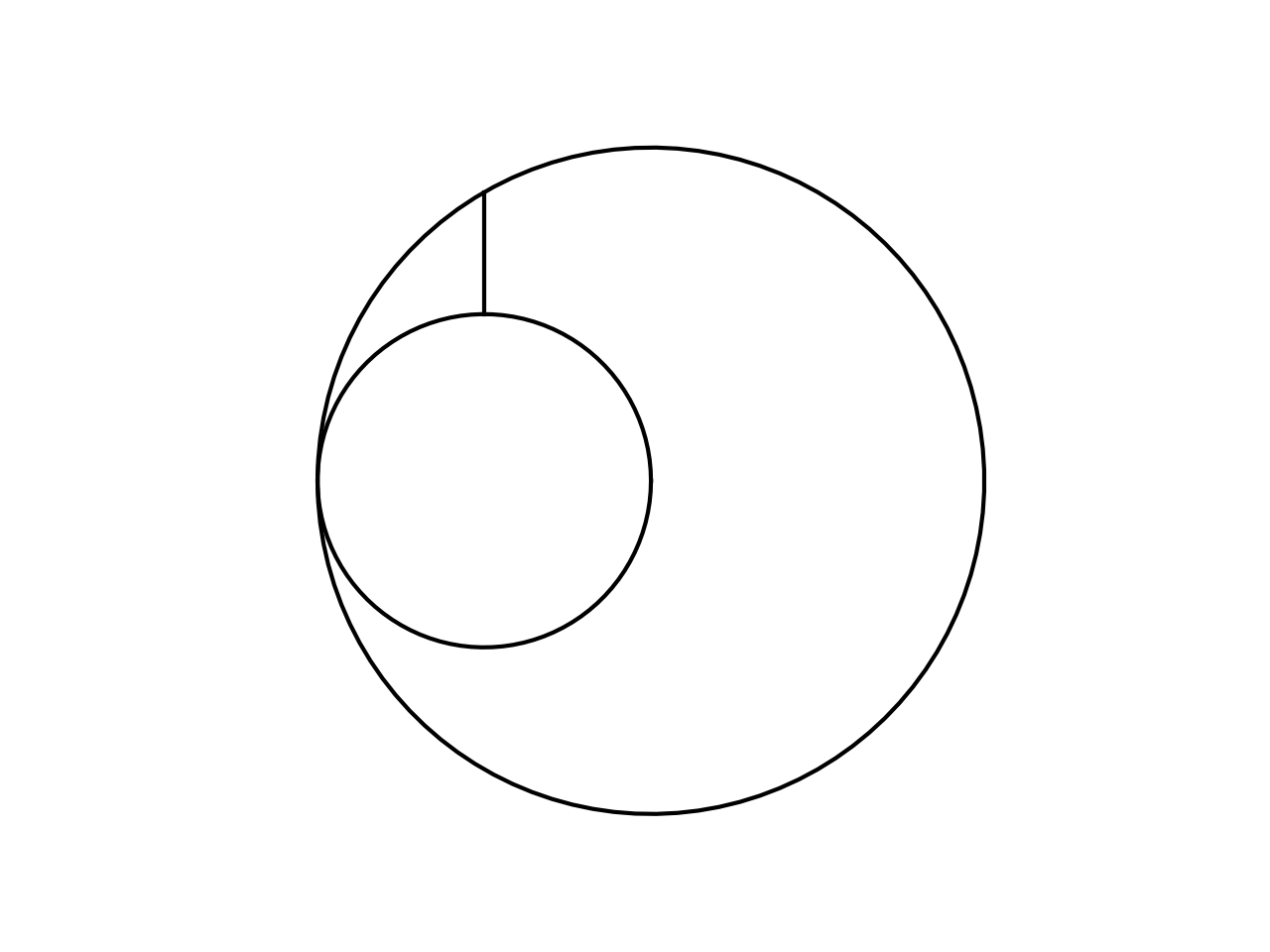}
      \put(75,51.5){ $D_1\cap D_2=T\cup T'$}

       \put(31.4,51.5){ $T$}
           \put(45,51.5){ $T'$}
                \end{overpic}
  \vspace{-1cm}
\caption{A circular arc triangle $T$ with angles $0$, $\pi/2$, $\pi/3$, and its complementary triangle $T'$.}
\label{fig:comple}
\end{figure}

 In order to give a more intuitive description of the map $p$ based on Theorem~\ref{main}, we proceed as follows. Note that the triangle $T_0$ has angles $0$, $\pi/2$, $\pi/3$
at its vertices $\infty$, $i$, $\rho\coloneqq\tfrac12 (1+i\sqrt3)$, respectively.
 Let $T\sub \bC$  be an arbitrary circular arc triangle in the Riemann sphere with angles equal to $0$, $\pi/2$, $\pi/3$. Each such triangle is  M\"obius equivalent to $T_0$ or its complex conjugate $\overline {T_0}$ (see Lemma~\ref{lem:triuniq} below). This implies that the triangle $T$ is contained in the  intersection $D_1\cap D_2$ of two closed disks $D_1$ and $D_2$ in $\bC$ each bounded by a circle in $\bC$ containing  one of the  sides of  $T$ and such that  these boundaries touch at the vertex of $T$ with angle $0$.  Then  the closure $T'$ of  the set 
 $(D_1\cap D_2) \setminus T$ 
is a circular arc triangle  that we call {\em complementary} to $T$, since $D_1\cap D_2=T\cup T'$. Note that $T'$  has the same vertices as $T$, is bounded by arcs of the same three circles as $T$, and has angles $0$, $\pi/2$, $2\pi/3$ (see Figure~\ref{fig:comple}). This relation of $T$ and its complementary triangle 
$T'$ is invariant under M\"obius transformations and complex conjugation.

\begin{figure}
 \begin{overpic}[ scale=0.4
    ]{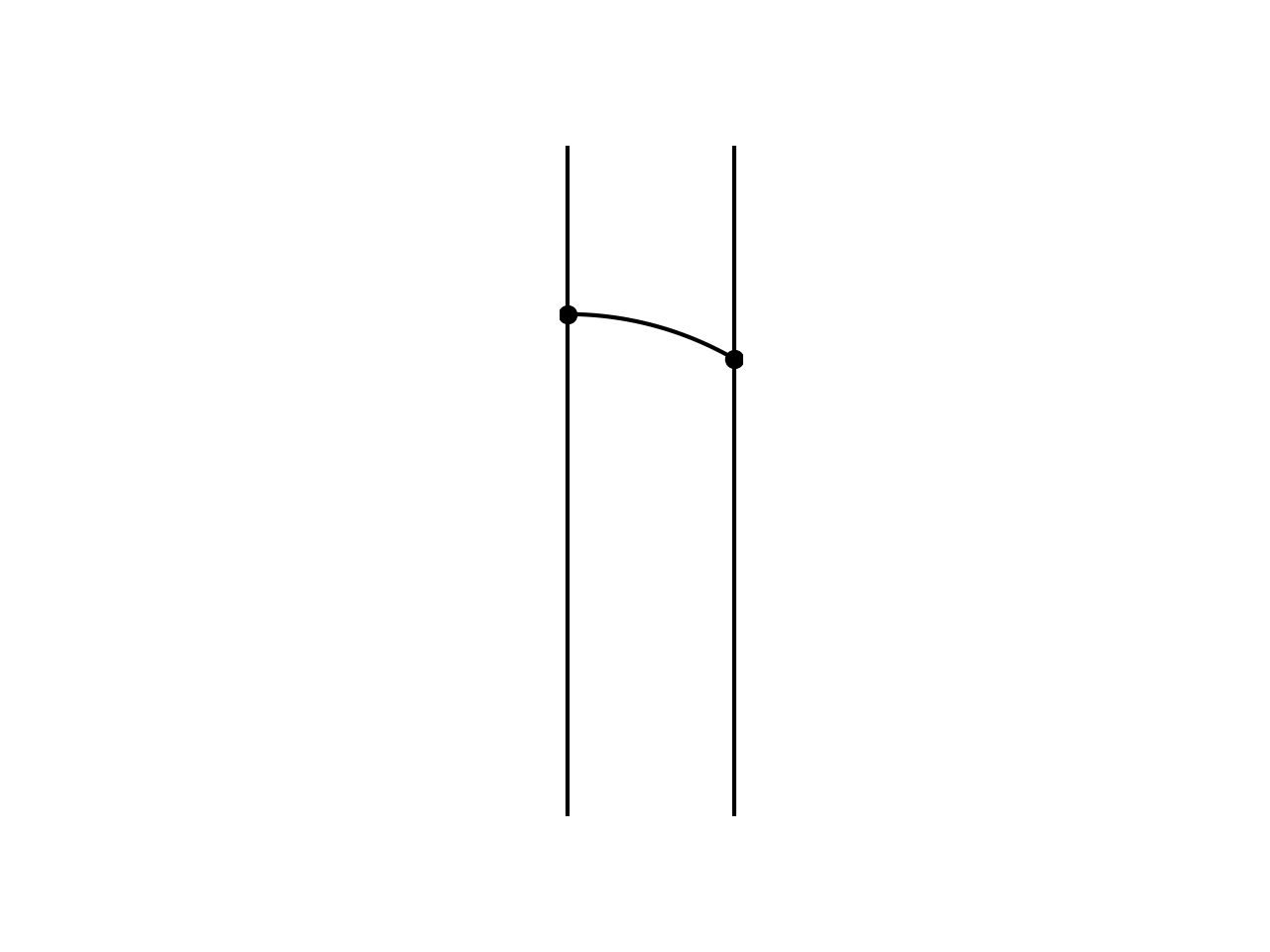}
       \put(47,40){ $\overline {T_1}$}
           \put(47,55){ $T_0$}
                   \put(39,49){ $i$}
    \put(58, 45.5){ $\rho$}

           \end{overpic}
  \vspace{-1cm}
\caption{The  circular arc triangle $T_0$ with its  complement $\overline{T_1}$.}
\label{fig:bar}
\end{figure}
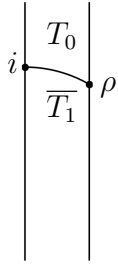

Note that $\overline {T_1}$ is the complementary triangle to $T_0$ in this sense (see Figure~\ref{fig:bar}). In particular, the complex conjugate  $\tau \mapsto 
\overline p(\tau)\coloneqq\overline {p(\tau)}$ of $p$ maps $T_0$ anti-conformally onto its complementary triangle 
$$ T_0'=\overline {T_1}=\{\infty\} \cup \{ \tau\in\C: 0\le \re(\tau) \le 1/2\}\setminus 
\{ \tau\in\C: \im(\tau)>0,\, |\tau|>1\} 
$$ 
such that the common vertices $i$, $\rho$, $\infty$ are fixed under $\overline p$.

Now it is well known that by successive reflections in the sides of $T_0$ we generate a tessellation $\mathcal{T}$ of the set 
$\H^*\coloneqq\H\cup \Q\cup\{\infty\}$
 by circular arc triangles with angles $0$, $\pi/2$, $\pi/3$ (see \cite[Section 2.2]{Sch}) and Figure~\ref{fig:T} for an illustration). 
 \begin{figure}[h]
 \begin{overpic}[ scale=0.55
    ]{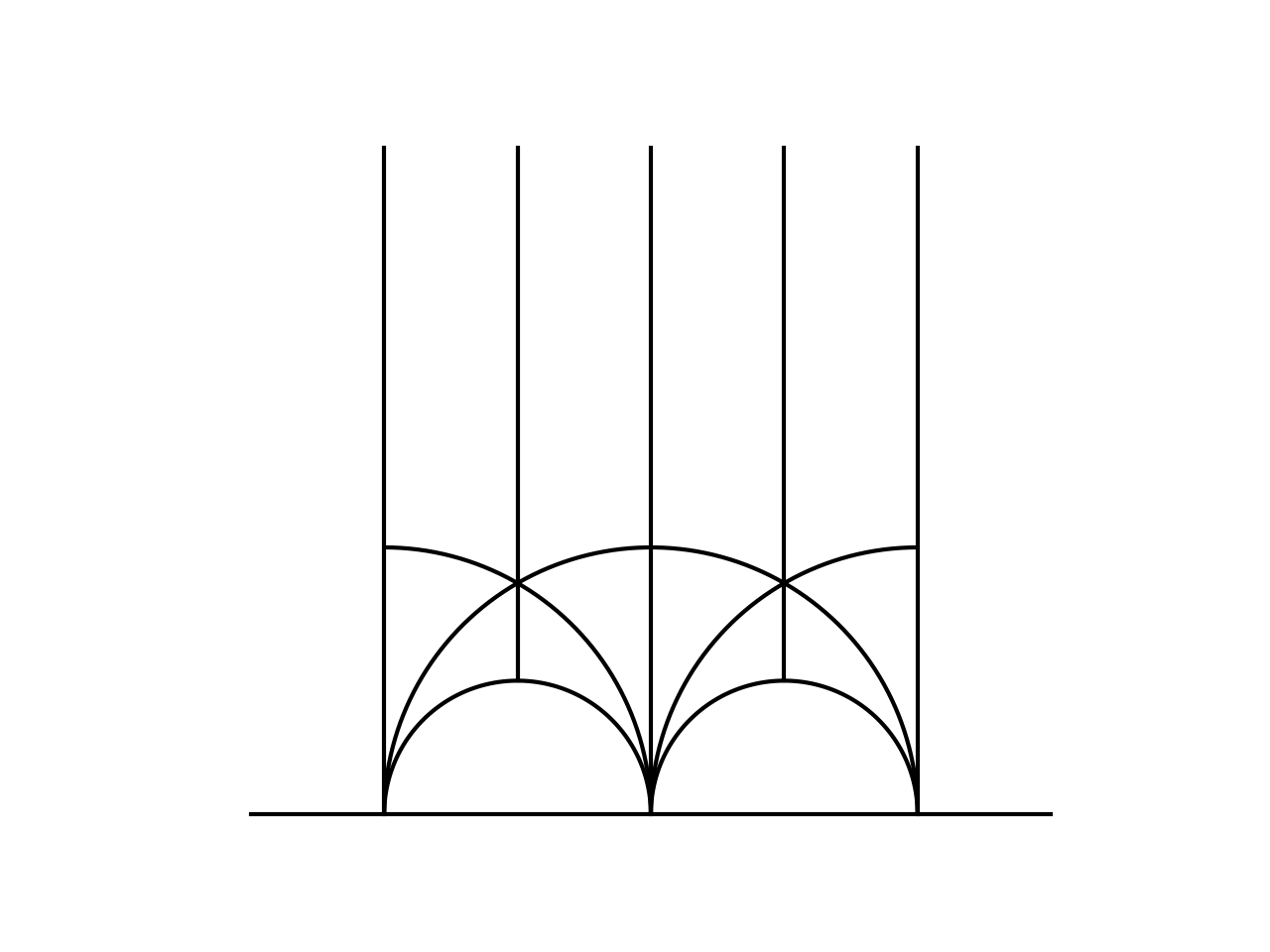}
    \put(51,33){ $i$}
    \put(61.5,31.5){ $\rho$}
     \put(72,12){ $1$}
       \put(52,12){ $0$}
        \put(52,45){ $T_0$}
           \end{overpic}
  \vspace{-1cm}
\caption{Some circular arc triangles in the tessellation 
$\mathcal{T}$.}
\label{fig:T}
\end{figure}
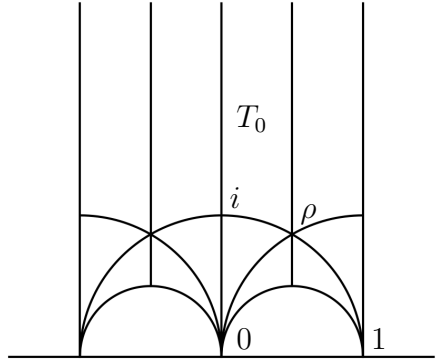

  Let  $G$ be the group of conformal or anti-conformal homeomorphisms of $\bC$ generated  by the reflections in the sides of $T_0$.   By applying the Schwarz Reflection Principle repeatedly and the invariance of the relation between a triangle and its complementary triangle  under this procedure, we immediately 
  see that $\overline p \circ S =S\circ \overline p$ for all $S\in G$. This gives the following consequence.

\begin{cor}\label{TT}Let $T$ be any triangle  in the tessellation $\mathcal{T}$ of $\H^*$ obtained by successive reflections in the sides of $T_0$, and let $T'$ be the complementary triangle of $T$ as defined above. Then $\overline p$ is an anti-conformal map of $T$ onto $T'$ that fixes the vertices of $T$. 
\end{cor}

 The M\"obius transformations of the form 
 \begin{equation}\label{Smodgr}
\tau\mapsto  S(\tau) = \frac{a\tau+b}{c\tau+d},
\end{equation}
 where $a,b,c,d\in \Z$ with $ad-bc=1$ form the {\em modular group}
 $\PSL_2(\Z)$.  This is a subgroup of index $2$ of the group $G$ defined above. 
 
 It follows from our previous considerations that
 \begin{equation}\label{transpS}
 p\circ S=S\circ p  \end{equation}
 for all $S\in \PSL_2(\Z)$.
 Meromorphic functions $f$ with this type of transformation behavior ($f\circ S=\widetilde S \circ f$, where $S$ runs through a  Fuchsian group $\Gamma$ and $\widetilde S$ is a M\"obius transformation associated with $S$)  are called  {\em polymorphic} in the classical literature. 
 We will refer to property \eqref{transpS} of $p$ as its {\em $\PSL_2(\Z)$-equivariance}.

 We will soon see that this equivariant  behavior of $p$ can easily be derived analytically
 (see Proposition~\ref{prop:poly}). Theorem~\ref{main} and  Corollary~\ref{TT} essentially explain this behavior from a geometric perspective.

The circular arc triangle 
\begin{equation}\label{eq:V0}
V_0=\{\tau \in \H: 0\le \re(\tau)\le 1,\, |\tau-1/2|\ge 1/2\}\cup\{\infty\}
\end{equation} 
has all its angles equal to $0$. It consists of a union of six triangles from the tessellation $ \mathcal{T}$ (see Figure~\ref{fig:pbar}). It is easy to see from Corollary~\ref{TT} that the contin\-uous extension of $\overline p$  to $V_0$ attains every value in $\bC$  
once, twice, or thrice. This is illustrated in Figure~\ref{fig:pbar}, where the level of darkness indicates how often the regions are 
covered (the darker the gray, the more often the region is covered; note that some relevant circles are drawn in black which does not correspond to how often the points in these circles are attained).  
In particular,    $\overline p(V_0\cap\H)=
\bC=p(V_0\cap\H)$. 

\begin{figure}[t]
 \begin{overpic}[ scale=0.7
    ]{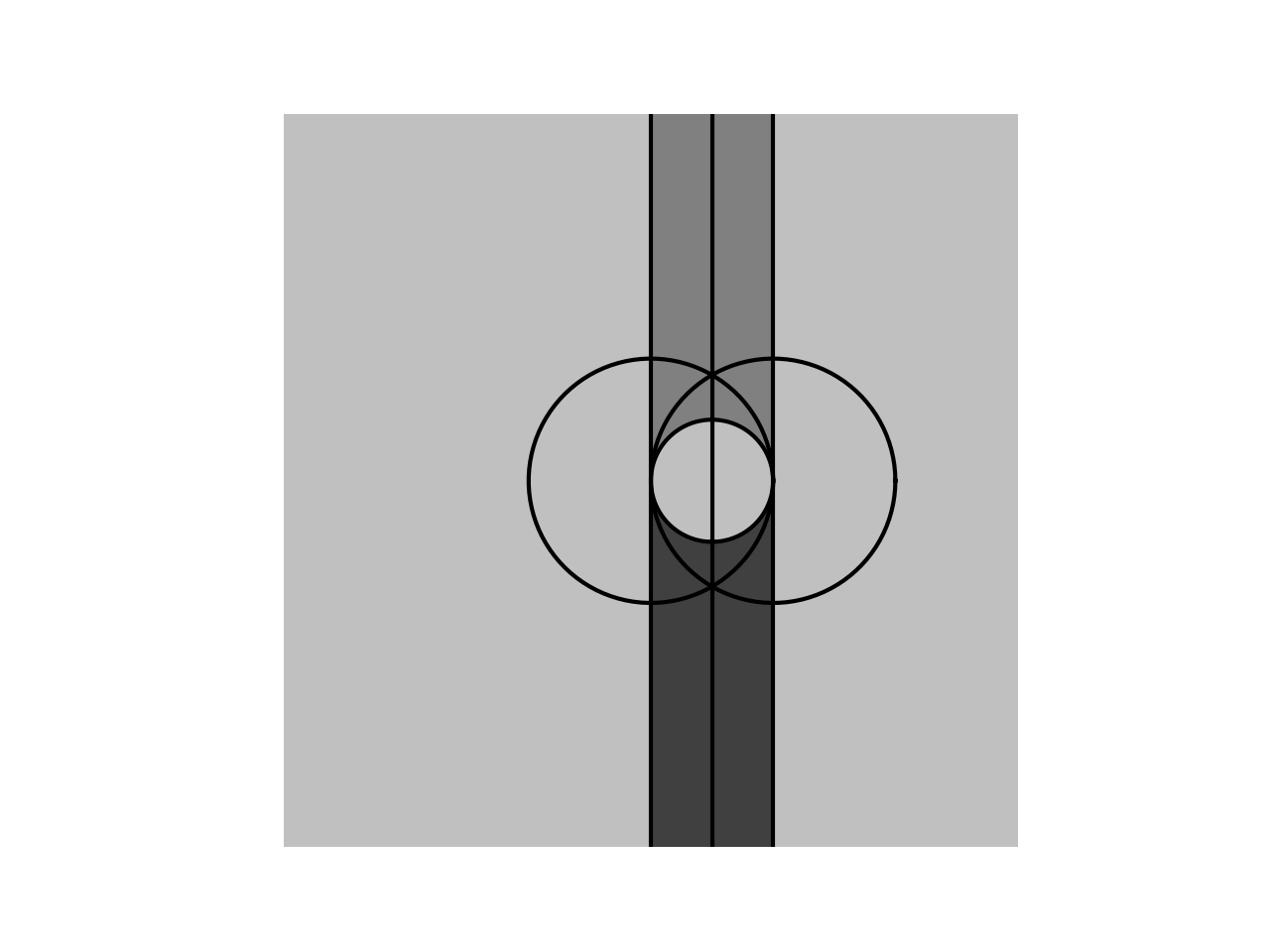}
                \put(46,55){ $V_0$}
                   \put(51,48){ $i$}
                     \put(56,47){ $\rho$}
                    \put(51,36){ $0$}
                     \put(61,36){ $1$}

           \end{overpic}
  \vspace{-1cm}
\caption{Sets covered by the map $\overline p$ restricted to $V_0$ .}
\label{fig:pbar}
\end{figure}
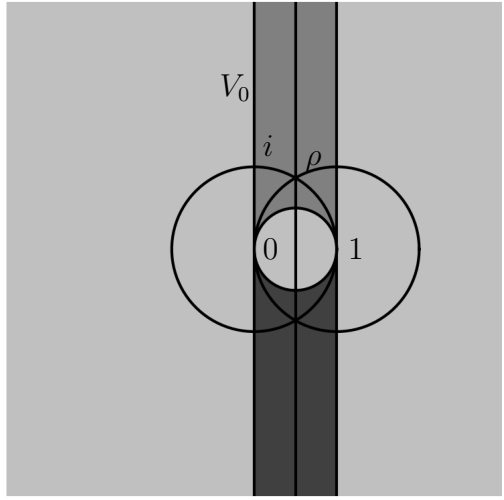

Now if we successively reflect in the sides of $V_0$, we  obtain  another tessellation  $\mathcal{T}'$ of $\H^*$ by circular triangles whose angles are all equal to $0$. 
Each of the  triangles $V$ in this tessellation is a union 
of six triangles in  $ \mathcal{T}$. Then Corollary~\ref{TT}
and the previous discussion imply that $p(V\cap\H)=\bC$. 
 We obtain the following consequence.
\begin{cor}\label{He} The map $p$ attains every value in $\bC$ infinitely often on $\H$. 
\end{cor} 
This was the main result in the papers \cite{He66, He67} by Heins. 
He proved it by a purely function-theoretic argument without 
 an explicit description of the map~$p$.

 We will give two proofs of Theorem~\ref{main}. The first one is along classical lines.  We suitably normalize the 
 quasi-periods $\eta_1$ and $\eta_2$ and  show that these normalized periods form a fundamental system of  
  solutions of a hypergeometric differential equation with a certain elliptic modular  function,
  the {\em absolute invariant $J$}, as the variable (see \eqref{etaODE}). This will show that $p=p(J)$ is a Schwarz triangle map sending the closed lower $J$-halfplane considered as a circular arc triangle with vertices $0$, $1$, $\infty$ to a circular arc triangle with angles $ 2\pi/3,\pi/2,0$.  
On the other hand, it is known that $J$ is a conformal map of  the circular arc triangle $T_0$ onto the  lower $J$-halfplane. This will lead to the  explicit description of the map $\tau \mapsto p(\tau)$.

Our second approach to Theorem~\ref{main} is more direct: we will show that $p$ is a homeomorphism from $\partial T_0$ onto $\partial T_1$. The statement then essentially follows from the Argument Principle (see Proposition~\ref{prop:arg}). 

We should emphasize that none of our results are really new. The hyper\-geometric ODE 
\eqref{etaODE}  for the suitably normalized quasi-periods $\eta_1$ and $\eta_2$ goes 
back to at least as far as  Fricke--Klein. The quasi-periods  of $\zeta$ are really the periods of the elliptic integral
 of the second kind in Weierstrass normalization (see \cite[p.~152 ff.\ and p.~198 ff.]{Fr} for a thorough discussion).  The ODE valid for them (in suitable normalization) is an example of what is now called a  Picard--Fuchs equation. They are satisfied  by periods of other elliptic and more general integrals (see \cite{Zag18} for a gentle introduction). Fricke--Klein  were also well aware of the connection to Schwarz triangle maps, but they seem not to have stated Theorem~\ref{main} explicitly even though it easily follows from their considerations. 

In the more recent literature, the function $p$ (in the form \eqref{pE2} below) was considered by various authors (see, for example, \cite{SS12, IJT, WY14}). In \cite{SS12}  the connection to hypergeometric ODEs was realized and an explicit formula for the Schwarzian derivative $\{p,\tau\}$ was obtained (see Section~\ref{s:rems} for more discussion), but the explicit geometric description of $p$ was not pointed out. 
 In  \cite{WY14} some mapping properties of $p$ were  studied without the full 
 realization of the statement of Theorem~\ref{main}. Actually, the authors of \cite{WY14} explicitly ask whether $p$ is a conformal map on the fundamental region 
  of the modular group given by the union of $T_0$ and its reflection in the imaginary axis.  This immediately follows from Corollary~\ref{TT}.
In \cite{IJT} the authors do prove a biholomorphism property of $p$ that is equivalent to Theorem~\ref{main}. Their method is somewhat ad hoc, based on similar considerations as  our second argument for the proof of Theorem~\ref{main}.

This paper is of an expository nature giving an introduction to this  
 classical subject. 
 An interesting topic for further investigation is whether mapping properties similar to Theorem~\ref{main} can also be obtained for  ratios of periods  of other naturally appearing elliptic integrals.   It is to be hoped that this paper is a starting point   for 
  further studies along these lines.

\medskip

\noindent
{\bf Acknowledgements.} I would like to thank Alex Eremenko for some interesting discussions and comments on this paper, and in particular for bringing  references \cite{He66, He67} to my attention which motivated the present investigation,  Bill Duke for interesting discussions and pointing out reference \cite{Wat},  Pietro Poggi-Corradini for bringing several typos to my attention, and the anonymous referee for a careful reading of the paper that led to many improvements.

  \section{First facts about the quasi-periods}
 We first compute some special values of $p$ and then derive a representation of $p$ in terms of a Fourier series. All of this is well known from the classical literature (see \cite{Fr}, for example). A modern account is given in  \cite{Cha}. 
 
 If $\om_1=i$ and $\om_2=1$, 
 then  $\Gamma=\{ki+n:k,n\in \Z\}$ and so $i\Gamma=\Gamma$. This implies that for  
  $\zeta(u)=\zeta(u;\Gamma)$, we have 
 $$ \zeta(iu; \Gamma)=\zeta(iu; i\Gamma)= -i\zeta(u; \Gamma).$$ 
 Hence
 $$ \eta_1=2\zeta(i/2)=-2i\zeta(1/2)=-i\eta_2, $$
 and so
 \begin{equation}\label{p(i)}
p(i)=\eta_1/\eta_2=-i.
\end{equation}

Throughout this paper, we use the notation
\begin{equation}\label{rho}
\rho=\tfrac{1}{2} (1+i\sqrt 3).
\end{equation}
Now let  $\om_1=\rho$, $\om_2=1$, $\Gamma=\{k\rho+n:k,n\in \Z\}$. Then
$$\rho\Gamma=\Gamma$$ which implies 
$$ \zeta(\rho u)= \zeta(\rho u; \Gamma)=\zeta(\rho u; \rho \Gamma)=\rho^{-1} \zeta(u; \Gamma)=\rho^{-1} \zeta(u)=\overline \rho \zeta(u).
$$ 
Hence 
$$ \eta_1=2\zeta(\rho/2 )=2\overline \rho \zeta(1/2 )=\overline \rho \eta_2,$$    
and 
  \begin{equation}\label{prho}
p(\rho)=\eta_1/\eta_2=\overline \rho.
\end{equation}

Now let $\om_1$ and $\om_2$ again be arbitrary,
$\Gamma$ be the associated  rank-2 lattice as in \eqref{latt}, and 
$\zeta(u)=\zeta(u;\Gamma)$. If we integrate $\zeta$ over the contour given by the parallelogram $Q$ with vertices
$ \pm \om_1/2\pm \om_2/2$, in  positive orientation, then the residue theorem gives us the {\em Legendre relation}    
\begin{equation}\label{Legr}
 2\pi i=\om_1\eta_2-\om_2\eta_1. 
\end{equation}
See \cite[pp.~50--51]{Cha} for more details. 

We will now prove a much deeper fact that connects the quasi-periods with the theory of modular forms which we will review in  Section~\ref{s:modforms}. 
We first set up some notation. As before, $\tau\in \C$ with $\im(\tau)>0$  will denote a variable in the open upper halfplane $\H$.
 Throughout this paper, we set 
 \begin{equation} q\coloneqq e^{2\pi i \tau}. \end{equation}
 Note that $|q|<1$ for $\tau\in \H$. 
We define
\begin{equation}
\sigma_k(n)\coloneqq \sum_{m=1,\, m|n}^nm^k
\end{equation}
 for $k,n\in \N$.
So $\sigma_k(n)$  is the sum of all $k$th-powers $m^k$ of natural numbers $m$ that divide $n$. 
Finally, we define 
  \begin{equation}\label{E2} E_2(\tau)\coloneqq 1-24\sum_{n=1}^\infty  \frac{nq^n}{1-q^n}=1-24\sum_{n=1}^\infty \sigma_1(n)q^n.
  \end{equation}
It is immediate that this series (as a function of $\tau\in \H)$ converges locally uniformly on $\H$. Hence $E_2$ is a holomorphic function on $\H$. 
\begin{prop} \label{prop:eta} For given $\om_1,\om_2$ let  $\eta_1,\eta_2$ be the quasi-periods of the associated $\zeta$-function. Then 
\begin{align}
 \eta_1&= -\frac{2\pi i}{\om_2}+\frac{\pi^2\om_1} {3 \om_2^2} E_2(\tau),\label{eta1}\\
 \eta_2 &= \frac{\pi^2} {3 \om_2} E_2(\tau). \label{eta2}
\end{align}
In particular,
\begin{equation}\label{pE2}
 p(\tau)=\eta_1/\eta_2=\tau-\frac{6i}{\pi E_2(\tau)}.
\end{equation}
\end{prop}\label{etaformu} 
Since $E_2(\tau) \to 1 $ as $\im(\tau)\to +\infty$, it follows from \eqref{pE2}
that 
\begin{equation}\label{infty}
p(\tau)\to \infty \text{ as } \im(\tau)\to +\infty. 
\end{equation}

 Formulas \eqref{eta1}  and \eqref{eta2} are well known; see, for example, 
 \cite[p.~311]{Fr}. Since they are the basis for our further considerations, we will give a proof for the sake of completeness. We will follow an argument outlined  in  
  \cite[pp.~262--263]{Fr}.
 \begin{proof} It is enough to show \eqref{eta2}. Then \eqref{eta1} follows from this and the Legendre relation  \eqref{Legr}, and \eqref{pE2} is an immediate consequence of \eqref{eta1} and  \eqref{eta2}.
 
 Our starting point is  formula \eqref{zeta}. For ease of notation we use the abbreviation $(k,n)=k\om_1+n\om_2$ and set 
 $$\Sigma_k(u)\coloneqq \sum_{n\in \Z}
\biggl( \frac{1}{u-(k,n)}+\frac1{(k,n)}+\frac{u}{(k,n)^2}\biggr)$$
for $k\in \Z\setminus\{0\}$. Moreover, we set
$$\Sigma_0(u)\coloneqq\frac{1}u+\sum_{n\in \Z\setminus\{0\}}
\biggl( \frac{1}{u-(0,n)}+\frac1{(0,n)}+\frac{u}{(0,n)^2}\biggr)$$ 
 Since the series \eqref{zeta} converges absolutely and uniformly in $u\in \C\setminus \Gamma$, we can rearrange the order of summation; so 
 $$ \zeta(u)=\sum_{k\in \Z}\Sigma_k(u).$$
 The idea now is to express $\eta_2=\zeta(u+\om_2)-\zeta(u)$ in terms of the differences of the functions $\Sigma_k$.
 First note that
 $$ \Sigma_0(u+\om_2)-\Sigma_0(u)=\frac{2}{\om_2}\sum_{n=1}^\infty\frac{1}{n^2}=
 \frac{\pi^2}{3\om_2}. $$ 
 For the other differences,  we use the standard fact that
 \begin{equation}\label{sin}
\frac{\pi^2}{\sin^2(\pi u)}=\sum_{n\in \Z}\frac{1}{(u-n)^2}
\end{equation}
for $u\in \C\setminus \Z$.
By considering partial sums, we see  that for $k\in \Z\setminus\{0\}$, we have
 \begin{align*}
 \Sigma_k(u+\om_2)-\Sigma_k(u)&=\sum_{n\in \Z}\frac{\om_2}{(k,n)^2}
 =\frac{1}{\om_2}\sum_{n\in \Z}\frac{1}{(k\tau+n)^2}\\
 &=\frac{\pi^2}{\om_2 \sin^2(|k| \pi \tau )}=\frac{-4\pi^2 q^{|k|}}{\om_2(1-q^{|k|})^2}
\\ &=-\frac{4\pi^2}{\om_2}\sum_{n=1}^\infty nq^{n|k|}. 
 \end{align*}
 Putting this all together, we see that
 \begin{align*} \eta_2&=\zeta(u+\om_2)-\zeta(u)=\sum_{k\in \Z}(\Sigma_k(u+\om_2)-\Sigma_k(u))\\
 &= \frac{\pi^2}{3\om_2}-\frac{8\pi^2}{\om_2}\sum_{k=1}^\infty\sum_{n=1}^\infty nq^{nk}=\frac{\pi^2}{3\om_2}-\frac{8\pi^2}{\om_2}\sum_{n=1}^\infty \sigma_1(n)q^n
 = \frac{\pi^2}{3\om_2}E_2(\tau),
 \end{align*} 
 as desired.
 \end{proof}
 We now give an analytic argument for the equivariant nature of $p$ under the modular group and a closely related transformation behavior of the function $E_2$.

 \begin{prop}\label{prop:poly} Let $a,b,c,d\in \Z$ with $ad-bc=1$.
 Then 
 \begin{equation}\label {polyp}
p\biggl(\frac{a\tau+b}{c\tau+d}\biggr) =\frac{ap(\tau)+b}{cp(\tau)+d}
\end{equation}
and 
\begin{equation}\label{transE2}
E_2\biggl(\frac{a\tau+b}{c\tau+d}\biggr)=  (c\tau+d)^2E_2(\tau)-\frac{6i}{\pi}c(c\tau+d).
\end{equation}
 \end{prop}
  \begin{proof} 
We define $\om_1'=a\om_1+b\om_2$ and $\om_2'=c\om_1+d\om_2$.
  Then $\om_1',\om_2'$ generate the same lattice $\Gamma$ as  $\om_1, \om_2$,
  and so give rise to the same $\zeta$-function. The first quasi-period $\eta'_1$ associated with the pair $\om_1',\om'_2$ can be computed from the following identity in $u$:
  $$\eta_1'=\zeta(u+\om_1')-\zeta(u)=\zeta(u+a\om_1+b\om_2 )-\zeta(u)=a\eta_1+b\eta_2. $$
 Similarly, for the second quasi-period $\eta_2'$ associated with $\om_1',\om'_2$, we have $\eta_2'=c\eta_1+d\eta_2$. 
  Equation \eqref{polyp} then immediately follows by passing to quotients.
  
  By \eqref{eta2} we have 
  $$ E_2(\tau)=3\eta_2\om_2/\pi^2.
  $$
  So if we use the previous notation and write $\tau'\coloneqq \om'_1/\om_2'=(a\tau+b)/(c\tau+d)$, then we have 
  \begin{align*}E_2(\tau')&=3\eta'_2\om'_2/\pi^2=3(c\eta_1+d\eta_2)(c\om_1+d\om_2)/\pi^2\\&=(c p(\tau)+d)(c\tau+d)E_2(\tau)
  =(c (\tau-6i/(\pi E_2(\tau)))+d)(c\tau+d)E_2(\tau)\\ &= 
  (c\tau+d)^2E_2(\tau)-\tfrac{6i}{\pi}c(c\tau+d).\qedhere
 \end{align*}
 \end{proof}
 
  A curious consequence of Corollary~\ref{He} is that there are pairs $(\om_1, \om_2)$ for which $\eta_1$ or $\eta_2$ vanishes (they cannot both be zero for a given pair $(\om_1, \om_2)$). In this case, $\zeta$ is a periodic, but not doubly-periodic function. This seems to have been first noticed by Watson \cite{Wat} and  was also pointed out by Heins in \cite{He66}.

  Note that \eqref{polyp} implies that $\eta_1=0$ for a pair $(\om_1, \om_2)$ if and only if $\eta_2=0$ for the pair $(-\om_2,\om_1)$.  On the other hand, by \eqref{eta2} we 
  have $\eta_2=0$ if and only if $E_2(\tau)=0$. So studying values $(\om_1, \om_2)$ where $\eta_1$ or $\eta_2$ vanishes amounts to the same as finding all zeros of $E_2$. As Heins  alluded to  in \cite{He67}, it is easy to see that $E_2$ has a zero on the positive imaginary axis. Watson \cite{Wat} had actually computed this zero with great accuracy,  The location of all the zeros of $E_2$ in $\H$ was studied in depth in 
  \cite{WY14, IJT}.

%

 \section{Modular forms}\label{s:modforms}
 In order to proceed, we need some  facts about {\em modular forms}. This is all standard material going back to Fricke--Klein \cite{FK}  and beyond. A modern account can be found in \cite{Sch}. The survey by Zagier \cite{Zag} gives a fresh perspective with connections to many other areas.   All we need are some  basic fundamentals of the theory. Many of them are discussed in \cite{Cha} and all of them  in \cite[Chapters 1--3]{Sch}, for example.   We start with a definition.
 
 \begin{defi} A holomorphic function $f$ on $\H$ is called a {\em (holomorphic and inhomogeneous)  modular form
 of weight} $k\in \N$ if 
 it satisfies
 \begin{equation}\label{mod1}
 f\bigg(\frac{a\tau+b}{c\tau+d}\biggr)=(c\tau+d)^kf(\tau)
\end{equation}
 for all $\tau\in \H$ and $a,b,c,d\in \Z$ with $ad-bc=1$.
 Moreover, we require that 
  \begin{equation}\label{mod2}
f(\tau)=O(1) \text{ as }  \im(\tau)\to +\infty.
\end{equation}
  \end{defi}
  Note that a modular form is $1$-periodic: $f(\tau+1)=f(\tau)$. 
  So it will have a Fourier series expansion
  \begin{equation}\label{Four}
f(\tau)=\sum_{n\in \Z} a_ne^{2\pi i n\tau}=\sum_{n\in \Z} a_n q^n
 \end{equation}
  that converges for all $q\in \C$ with $|q|<1$. Condition \eqref{mod1} ensures that no terms with negative powers of $q$ occur in this series and so 
  \begin{equation}\label{modFour} f(\tau)=\sum_{n\in \N_0} a_n q^n.
  \end{equation}
  
  If $f$ is a function on $\H$, $k\in \N$, and $S\in \PSL_2(\Z)$ is as in \eqref{Smodgr}, then we can define an operation $f|_kS$ by setting $(f|_kS)(\tau)=f(S(\tau))(c\tau+d)^{-k}$.
  Condition \eqref{mod1} then simply becomes $(f|_kS)(\tau)=f(\tau)$.

  One can easily check that if $S,T\in  \PSL_2(\Z)$, then
  $$f|_k(S\circ T)=(f|_kS)|_kT. $$
 This implies that in order to verify \eqref{mod1}, it is enough to do this for generators 
 of $ \PSL_2(\Z)$, for example for $\tau\mapsto \tau+1$ and $\tau\mapsto -1/\tau$.
 We see that condition   \eqref{mod1} is equivalent to the identities
 $$ f(\tau+1)=f(\tau) \text { and } f(-1/\tau)=\tau^kf(\tau).$$ 
  
  We say that $f$ is a {\em modular function} if it is meromorphic on $\H$
  and satisfies $f\circ S=f$ for all $S\in \PSL_2(\Z)$. Moreover, we require that that for some $N\in \N_0$ we have
 \begin{equation} \label{funcmodbd} f(\tau)=O(q^{-N}) \text{ as }  \im(\tau)\to +\infty.
 \end{equation}
A modular function $f$ is $1$-periodic and has a Fourier expansion as in \eqref{Four} converging
 if $|q|>0$ is small enough.
Condition \eqref{funcmodbd} ensures that in this  Fourier expansion only finitely many non-zero terms with negative $n$ occur. 

Note that if $g$ and $h\ne 0$ are modular forms of the same weight, then $f=g/h$ is a modular function. If $g$ and $h$ are modular forms of weights $k$ and $l$, respectively, then $gh$ is a modular form of weight 
$k+l$. 

Let $f$ be a modular form of weight $k$. Then we can pass to an associated 
{\em homogeneous modular form} $\widetilde f$ depending on two variables 
$\om_1$ and $\om_2$ (satisfying our standing assumptions) by setting
$$ \widetilde f(\om_1,\om_2)= \biggl(\frac{2\pi}{\om_2}\biggr)^k f(\om_1/\om_2)= \biggl(\frac{2\pi}{\om_2}\biggr)^k f(\tau). $$
The incorporation of the factor $2\pi$ here has some advantages. Then $ \widetilde f$  is a homogeneous function of degree $-k$ in the sense that 
$$   \widetilde f(t\om_1,t\om_2)=t^{-k}  \widetilde f(\om_1,\om_2) $$
for $t\in \C^*$. Moreover, \eqref{mod1}  is equivalent to 
$$ \widetilde f(a\om_1+b\om_2, c\om_1+d\om_2)= \widetilde f(\om_1, \om_2).$$
So the transformation behavior of a homogeneous modular form becomes more transparent at the cost of having to deal with a function of two variables.

The (inhomogeneous and holomorphic) modular forms of a given weight $k\in \N$ obviously form a vector space $\mathcal{M}_k$ over $\C$. It immediately follows from the definitions that there are no non-trivial modular forms of odd weight (indeed, change the signs of $a,b,c,d$ in \eqref{mod1}). 
If $k$ is even, then the following well-known fact gives the dimension of  $\mathcal{M}_k$ for 
even $k$  (see \cite[Theorem~18, p.~47]{Sch}). The fact that 
$\mathcal{M}_k$ is finite-dimensional is probably the single most important fact in the theory of modular forms.  
\begin{prop}  For even $k\in \N$ we have 
$$\dim \mathcal{M}_k=   \left\{ \begin{array} {cl} \lfloor k/12 \rfloor & \text{for 
$k\equiv 2$  \textnormal {mod} $12$,}\\ &\\
 \lfloor k/12 \rfloor+1&\text{for $k\not \equiv 2$  \textnormal {mod}  $12$.}  \end{array}
 \right.
$$
\end{prop}
In particular, there are no non-trivial modular forms of weight $2$. Moreover, the spaces of modular forms of weight $4$, $6$, and $8$ are all one-dimensional and hence spanned by
any non-trivial form in these spaces. The latter fact is all we need in the following.

To obtain non-trivial modular forms in $\mathcal{M}_4$ and $\mathcal{M}_6$,
we consider the {\em Eisenstein series} $G_k$ defined for even 
 $k\ge 4$ as 
 $$G_k(\tau)=\sum_{m,n\in \Z}\phantom{}^{\!\!\!\!'} \,\,\frac{1}{(m\tau+n)^k}. $$
  Here the prime on the sum means that the term with 
 $m=n=0$ should be omitted from the summation.

The series representing the function $G_k$, $k\ge 4$, converges absolutely and locally uniformly for $\tau\in \H$; so $G_k$ is a holomorphic function of $\tau\in \H$.  Actually, $G_k$ is a modular form of weight $k$. Indeed, it is immediate to see that $G_k$ has the right transformation behavior as in \eqref{mod1}. Moreover, one can explicitly 
obtain a Fourier expansion of $G_k$ as in \eqref{modFour}. This can be seen by a computation similar to the proof of 
Proposition~\ref{prop:eta}  where one uses an identity obtained from \eqref{sin} by differentiating $(k-2)$-times. This is  standard and the details can be found in \cite[Chapter~3]{Sch}, for example. We only need the result for $k=4$ and $k=6$. 
 Namely, we have 
$$ G_4=\frac{\pi^4}{45}E_4,\quad G_6=\frac{2\pi^6}{945}E_6, $$
where
  $$E_4=1+240 \sum_{n=1}^\infty  \frac{n^3q^n}{1-q^n}
 = 1+240 \sum_{n=1}^\infty \sigma_3(n)q^n, $$
  $$E_6=1-504 \sum_{n=1}^\infty  \frac{n^5q^n}{1-q^n}=
  1-504 \sum_{n=1}^\infty \sigma_5(n)q^n. $$
  
In particular, $G_4$, or equivalently $E_4$, is a modular form spanning $\mathcal{M}_4$, and
$G_6$ or $E_6$ are  modular forms each spanning $\mathcal{M}_6$. The space 
$\mathcal{M}_8$ is spanned by $E_4^2$.

While $E_2$ (as defined in \eqref{E2})  is not a 
modular form, it has a transformation behavior that is closely related.   
Indeed, each function $E_k$, $k=2,4,6$  is  $1$-periodic, that is, for $\tau\in \H$ it satisfies
 $$ E_k(\tau+1)=E_k(\tau);$$
 moreover, we have
 \begin{align} E_2(-1/\tau)&=\tau^2E_2(\tau)-\tfrac{6 i}{\pi}  \tau,\label{transE2'}\\
 E_4(-1/\tau)&=\tau^4E_4(\tau), \label{transE4}\\
 E_6(-1/\tau)&=\tau^6E_6(\tau). \label{transE6}
 \end{align}
 The first equation follows from \eqref{transE2},   while the last two equations follow from the fact that $E_4$ and $E_6$ are 
modular forms of weight $4$ and weight $6$, respectively.
Due to its simple transformation behavior,  $E_2$ is   called a  {\em quasi-modular form of weight} $2$.  


If, as before, $\rho=\frac{1}{2}(1+i\sqrt3)$, then
$$-1/ \rho=-\tfrac{1}{2}(1-i\sqrt3)=\rho-1,$$
and so 
$$ \rho^4E_4(\rho)=E_4(-1/ \rho)=E_4(\rho-1)=E_4(\rho).$$
Hence 
\begin{equation}\label{E4rho}
E_4(\rho)=0.
\end{equation}
It follows from the {\em valence formula}  for modular forms  (see \cite[Theorem 13, p.~41]{Sch}) that $\rho$ is the only zero of $E_4$ on 
$\H^*=\H\cup\Q\cup\{\infty\}$ up to equivalence under $\PSL_2(\Z)$. More precisely, we have
\begin{equation}\label{E4zeros}
E_4(\tau)=0 \text{ for } \tau\in \H^* \text { if and only if } \tau=S(\rho) \text{ for some } S\in \PSL_2(\Z).
\end{equation}

    If we insert $\tau=i$ in the transformation formula \eqref{transE6} for $E_6$, we see that 
$E_6(i)=-E_6(i)$, and so
\begin{equation}\label{E6i}
E_6(i)=0. 
\end{equation}
Again $i$ is the only zero of $E_6$ up to equivalence under $\PSL_2(\Z)$, but we will not need this fact.

We will derive   expressions for the derivates  $E'_k=dE_k/d\tau$. These formulas are often attributed to 
 Ramanujan, but they can be traced back (in different notation) to Fricke--Klein and beyond. 
 To absorb a factor $2\pi i$
that appears in these formulas, we introduce the abbreviation
 $$DF=\frac{1}{2\pi i} \frac{dF(\tau)}{d\tau}$$
 for ease of notation.  
 \begin{prop}\label{Ederi}
 The following identities  are valid:
\begin{align}  DE_2&=\tfrac{1}{12}(E_2^2-E_4),\label{E2der}\\
 DE_4&=\tfrac{1}{3}(E_2E_4-E_6),\label{E4der}\\
 DE_6&=\tfrac{1}{2}(E_2E_6-E_4^2)\label{E6der}. 
 \end{align}
 \end{prop}
 \begin{proof} To prove the first formula,  we consider the  function function $f=DE_2-\tfrac{1}{12}E_2^2$ on $\H$. 
 It is $1$-periodic and has a Fourier expansion of the form \eqref{modFour} as follows from computations with the underlying $q$-power series for $E_2$.  Moreover, the transformation behavior of $E_2$ gives 
\begin{align*}  \tau^{-2} E'_2(-1/\tau)&
=\frac{d}{d\tau}E_2(-1/\tau)=\frac{d}{d\tau} 
(\tau^2E_2(\tau)-\tfrac{6 i}{\pi } \tau)\\
&=2\tau E_2(\tau)+\tau^2E_2'(\tau) -\tfrac{6 i}{\pi }, 
 \end{align*}
 and so 
 $$DE_2(-1/\tau)=\tau^4 DE_2(\tau)-\tfrac{i}{\pi} \tau^3 E_2(\tau)-\tfrac{3 }{\pi ^2}\tau^2. 
 $$
 Hence 
\begin{align*}
f(-1/\tau)&=DE_2(-1/\tau)-\tfrac{1}{12}E_2(-1/\tau)^2\\
&=\tau^4 DE_2(\tau)-\tfrac{i}{\pi} \tau^3 E_2(\tau)-\tfrac{3}{\pi^2 }\tau^2 -\tfrac{1}{12}(\tau^2E_2(\tau)-\tfrac{6 i}{\pi}  \tau)^2\\
&=\tau^4f(\tau).
\end{align*} 
It follows that $f$ is a modular form of weight $4$. Since the  space of these forms has dimension $1$ and is spanned by 
$E_4$, there exists a constant $c\in \C$ such that 
$f=cE_4$. To determine this constant, we consider the 
leading coefficients of the $q$-expansions of  $f$ and $E_4$. We find that 
$$ f(\tau)= -\tfrac1{12} +O(q), \quad E_4=1+O(q), $$
and so $c=-\tfrac1{12}.$ The claim follows. 

 The other identities follow from similar considerations:
 one shows that $DE_4-\frac13 E_2E_4$ is a modular form of weight $-6$  and 
 $DE_6-\frac12 E_2E_6$ is a modular form of weight $-8$.
 Moreover, the spaces of these forms are spanned by 
 $E_6$ and $E_4^2$, respectively.  Again one determines the proportionality constants by considering $q$-expansions. 
 \end{proof} 
 To put the preceding argument into a more general perspective:
if $f$ is a modular form of weight $k=4, 6, \dots$, then one can show by computations similar to the ones in the proof of Proposition~\ref{Ederi} that 
its {\em Serre derivative} (see \cite[p.~48]{Zag})
$$\vartheta_kf\coloneqq Df-\tfrac{k}{12}E_2f$$
is a modular form of weight $k+2$.

We need two more auxiliary functions. The first one is
$$\Delta\coloneqq\frac{1}{1728}(E_4^3-E_6^2). $$
This is a modular form of weight $12$ (note that the notation in \cite{Sch, Cha} is different from ours). From Proposition~\ref{Ederi} it follows that 
\begin{align*} D  \Delta&=\tfrac{1}{1728}(D(E_4^3)-D(E_6^2))=\tfrac{1}{1728}(3E_4^2DE_4-2E_6DE_6)\\
&=\tfrac{1}{1728}(E_4^3-E_6^2)E_2= \Delta E_2. 
\end{align*} 
Based on this and the expansion of $E_2$ as a $q$-series,
it is easy  to derive the well-known formula
\begin{equation}\label{Dprod}
\Delta =q\prod_{n=1}^\infty (1-q^n)^{24} 
\end{equation} 
Indeed, both sides in this identity are holomorphic functions on $\H$ with the  same logarithmic derivatives, namely $2\pi i E_2$; so they must represent the same function up to a multiplicative constant. This constant is equal to $1$ as a comparison of the leading terms of the $q$-expansions shows. 

It immediately follows from \eqref{Dprod} that $\Delta$ has no zeros on $\H$ and that it takes positive real values for $\tau$ on the positive imaginary axis. In the following, we will consider various roots $\Delta^{1/k}$ of $\Delta$ for $k\in \N$.
Since $\Delta$ has no zeros on $\H$, these are holomorphic functions on $\H$. A priori, $\Delta^{1/k}$ is only defined up to multiplication by a $k$-th root of unity. We fix this ambiguity so that
$\Delta^{1/k}$ attains positive real values on the positive imaginary axis.

We need one more auxiliary function, namely  the $J$-{\em invariant} given as 
\begin{equation}\label{defJ}
J=E_4^3/(E_4^3-E_6^2)= \tfrac{1}{1728} E_4^3/\Delta. 
\end{equation}

This is a modular function. It is well known that $J$  maps the circular arc  triangle $T_0$ as in  \eqref{T0} conformally to the closed lower halfplane, considered 
as a circular arc triangle with vertices $0$, $1$, $\infty$ (see 
\cite[Chapter 6; Theorem 5, p.~90]{Cha}). 

At least the correspondence of vertices is easy to see:
the definition of $J$ immediately gives
$$J(\tau)=\tfrac{1}{1728}q^{-1}+O(1) \text{ as } \im(\tau)\to +\infty. $$ 
Hence $J(\infty)=\infty$ (understood in a limiting sense). Moreover, the definition of 
$J$ in combination with  \eqref{E4rho} and   \eqref{E6i}  gives 
$$ J(\rho)=0  \text{ and } J(i)=1. $$ 

By \eqref{defJ} we obtain a uniquely determined holomorphic third  root of $J$ by setting $$J^{1/3}\coloneqq\tfrac 1{12}E_4/\Delta^{1/3}.$$
Similarly, by \eqref{defJ} we have $J-1=\tfrac{1}{1728}E_6^2/\Delta$, which 
allows us to define the holomorphic function 
$$(J-1)^{1/2}\coloneqq\tfrac{1}{24 \sqrt 3}E_6/\Delta^{1/2}. $$
Note that these definitions fix the ambiguity of $J^{1/3}$  and $(J-1)^{1/2}$  in such a way that these functions take positive real values with  $\tau$ on the positive imaginary axis with $\im(\tau)$ large.

The previous formulas can be rewritten as follows 
$$E_4=12J^{1/3} \Delta^{1/3}, \quad E_6=24\sqrt3 (J-1)^{1/2}
\Delta^{1/2}. $$ 
We also have
\begin{align} DJ&= \tfrac{1}{1728} \biggl( \frac{3 E_4^2 D(E_4)}{ \Delta}- \frac{E_4^3D( \Delta)}{\Delta^2}\biggr) \label{eq:derJ}\\ &=
- \tfrac{1}{1728}E_4^2E_6/ \Delta=- 2\sqrt3 J^{2/3} (J-1)^{1/2} \Delta^{1/6} . \notag
 \end{align} 
 
 We want to introduce $J$ as a new variable instead of $\tau$ for some functions of $\tau$. The map $\tau\in \H\mapsto J(\tau)$ is locally injective 
 at all points not equivalent to $i$ or $\rho$ under the modular group. The latter points  are exactly those where $J$ takes on the values $1$ or $0$. Moreover, $J(\tau)$ attains  all values in $\bC$ except $\infty$. This means that $\tau=\tau(J)$ is locally well defined away from the points $0$, $1$, $\infty$ in the $J$-plane. Once we fix such a local branch $\tau=\tau(J)$, we can analytically continue it along any path in $\bC\setminus\{0,1,\infty\}$. In this  way, we obtain a multi-valued function on $\bC\setminus\{0,1,\infty\}$, whose branches differ by postcomposition with elements in the modular group $\PSL_2(\Z)$. 

We will now fix such a local branch $\tau=\tau(J)$ and introduce the following functions, considered as depending on the variable  $J$:  
\begin{align} \Om_1&= \tau\Delta^{1/12},\quad \Om_2= \Delta^{1/12},\label{Om}
\\
H_1&=\frac{1} { \Delta^{1/12}}(\tau E_2-\tfrac{6i}{\pi}), \quad  
  H_2=\frac{E_2} { \Delta^{1/12}} \label{H}.
 \end{align} 
 Note that by \eqref{pE2} we have 
 \begin{equation} \label{eq:Hrat} 
H_1(\tau)/H_2(\tau)=\tau-\frac{6i}{\pi E_2(\tau)}=p(\tau). 
\end{equation}
 The meaning of the expressions \eqref{Om} and \eqref{H}  becomes clearer if one transitions to homogeneous functions of the variables $\om_1$ and $\om_2$. See Section~\ref{s:rems} for more details.

\begin{lem}\label{OHderiv} For $k=1,2$ we have the following identities: 
\begin{align*}  \frac{d\Om_k}{dJ}&=-\tfrac{1}{24\sqrt3} J^{-2/3} (J-1)^{-1/2} H_k,\\
\frac{dH_k}{dJ}&= \tfrac{1}{2\sqrt3} J^{-1/3} (J-1)^{-1/2}  \Om_k.
\end{align*} 
\end{lem}

\begin{proof} We have 
$$
 D\Om_1=\tfrac1{12}  \Delta^{1/12} \tau E_2 
-\tfrac{i}{2\pi}  \Delta^{1/12}, \quad D\Om_2=\tfrac1{12}  \Delta^{1/12} E_2,
$$ 
Hence 
\begin{align*}  \frac{d\Om_2}{dJ}&= D\Om_2/DJ \\
&=- \tfrac{1}{24\sqrt3} E_2 J^{-2/3} (J-1)^{-1/2}/\Delta^{1/12}\\
&=-\tfrac{1}{24\sqrt3} J^{-2/3} (J-1)^{-1/2} H_2. 
\end{align*} 

Similarly,
 \begin{align*}\frac{d\Om_1}{dJ}&= D\Om_1/DJ=
  \frac{\tfrac1{12}  \Delta^{1/12} \tau E_2 
-\tfrac{i}{2\pi}  \Delta^{1/12}}
{-2\sqrt3 J^{2/3} (J-1)^{1/2} \Delta^{1/6}}\\
&=-\tfrac{1}{24\sqrt3}  J^{-2/3} (J-1)^{-1/2} H_1.
 \end{align*} 
 The first identity follows. 
 
 For the second identity note that 
\begin{align*} DH_2&=
 D(E_2/\Delta^{1/12})= DE_2/\Delta^{1/12}-\tfrac1{12} E_2^2/\Delta^{1/12}\\
 &=-\tfrac1{12} E_4/\Delta^{1/12}=
 -J^{1/3}  \Delta^{1/4} 
  \end{align*} and so 
  $$\frac{dH_2}{dJ}= DH_2/DJ=
  \frac{-J^{1/3} \Delta^{1/4}}{- 2\sqrt3 J^{2/3} (J-1)^{1/2}\Delta^{1/6}}= \tfrac{1}{2\sqrt3}
  J^{-1/3} (J-1)^{-1/2} \Om_2.
  $$ 
 Similarly, 
\begin{align*}  DH_1&=
 D(\tau E_2-\tfrac{6i}{\pi})/\Delta^{1/12}-\tfrac1{12} (\tau E_2-\tfrac{6i}{\pi}) E_2/\Delta^{1/12}\\
 &= (-\tfrac{i}{2\pi}E_2+\tfrac{1}{12} \tau(E_2^2-E_4))/\Delta^{1/12}-\tfrac1{12} (\tau E_2-\tfrac{6i}{\pi}) E_2/\Delta^{1/12}\\
 &=-\tfrac1{12}\tau E_4/\Delta^{1/12}=-\tau J^{1/3} \Delta^{1/4}=\tau DH_2.
  \end{align*} 
  It  follows that 
 \begin{align*}\frac{dH_1}{dJ}&= DH_1/DJ=\tau DH_2/DJ
  = \tfrac{1}{2\sqrt3}
  J^{-1/3} (J-1)^{-1/2}\tau  \Om_2\\ & = \tfrac{1}{2\sqrt3}
  J^{-1/3} (J-1)^{-1/2}\Om_1.\qedhere
  \end{align*}
   \end{proof} 
   
  \smallskip 
\begin{prop}\label{prop:hypp}
The functions $H_1$ and $H_2$ considered as functions of 
$J$ form a fundamental system of solutions of 
the  hypergeometric differential equation  
\begin{equation} \label{etaODE}
w''(z)+\frac{5z-2}{6z(z-1)}w'(z)+\frac{1}{144z(z-1)}w(z)=0.
 \end{equation} 
\end{prop} 
As mentioned in the introduction, this is not new. It follows from the methods of Fricke--Klein and was explicitly stated in \cite[formula (9), p.~326]{Fr}. To derive
\eqref{etaODE},  Fricke used homogeneous modular forms and an associated differentiation process in contrast to our approach based on inhomogeneous modular forms (see Section~\ref{s:rems} for more discussion). 
\begin{proof} 
We use the system of equations from  Lemma~\ref{OHderiv}.
For simplicity we write $H=H_k$ and $\Om=\Om_k$ for $k=1,2$. 
Then
\begin{align*}
\frac{d^2H}{dJ^2}& =   \frac{d}{dJ}\biggl( \frac{1}{2\sqrt3}  J^{-1/3} (J-1)^{-1/2}  \Om\biggr)\\
&=\biggl(-\frac {1}{3J} -\frac {1}{2(J-1)}\biggr) \frac{dH}{dJ}+
\frac{1}{2\sqrt3}  J^{-1/3} (J-1)^{-1/2}  \frac{d\Om}{dJ}\\
&=-\frac{5J-2}{6J(J-1)}\frac{dH}{dJ}-\frac{1}{144J(J-1)}H. 
\end{align*} 
It follows that $H_1$ and $H_2$ are solutions of the above differential equation.  Since the ratio $p=H_1/H_2$ of these functions is non-constant by \eqref{pE2},  they are linearly independent  and  form a fundamental  system of solutions. 
   \end{proof} 
The general form of the hypergeometric differential equation  (see \cite [Part 7, Chapter 2]{Car} or \cite[Chapter 8]{Bie}) is 
\begin{equation}\label{hyperODE}
 w''(z)+\frac{z(1+\alpha+\beta)-\ga}{z(z-1)}w'(z)+\frac{\alpha\beta}{z(z-1)}w(z)=0. 
\end{equation}

So the parameters in \eqref{etaODE} are $\alpha=\beta=-1/12$ and $
\ga=1/3$. This implies that the ratio $p=H_1/H_2$ is a conformal map of the closed upper or lower $J$-halfplane  considered as circular arc triangles with vertices $0$, $1$, $\infty$ onto a circular arc triangle with angles $\la\pi$,  $\mu\pi$, $\nu\pi$, respectively, where   
$$ \la=1-\ga=2/3,\quad \mu=\ga-\alpha-\beta=1/2,\quad \nu=\alpha-\beta=0.$$ 
This is thoroughly explained in \cite[Part 7, Chapter 2]{Car}  or \cite[pp.~252--254]{Bie}.  

Before we proceed,  we record a simple geometric fact.
\begin{lem} \label{lem:triuniq}
Let $T\sub \bC$  be a circular arc triangle. If $T$ has angles $0, \pi/2, \pi/3$, then it is M\"obius equivalent
to  $T_0$ or $\overline {T_0}$. 
 If $T$ has angles $0, \pi/2, 2\pi/3$,  then it is M\"obius equivalent to $T_1$ or $\overline {T_1}$. 
\end{lem} 
Here we say that two circular arc triangles $X$ and $Y$  are {\em M\"obius equivalent} if there exists a M\"obius transformation $S$ (i.e., a biholomorphism on $\bC$) that gives a conformal map  between $X$ and $Y$ (as defined in the introduction).  Recall that $T_0$ and $T_1$ are defined in \eqref{T0} and \eqref{T1}, respectively.  

\begin{proof} The first part is true  in even  greater generality and its proof  is fairly standard 
(see, for example, \cite[Section 394 and Figure 83] {Car}); so we will only give an outline of the argument.

Suppose  $T\sub \bC$ has angles $0, \pi/2, \pi/3$ at its vertices $a,b,c\in \partial T$, respectively. By applying an auxiliary M\"obius transformation to $T$, we may assume that $a=\infty$.
Then the side  of $T$ containing $b$ and $c$ must be an
arc on a circle $C$ and cannot be a segment on a line (in the latter case the angles at $b$ and $c$ would  have to add up to $\pi$ which they do not). 
By applying another M\"obius transformation (namely, a Euclidean similarity), we may also assume that
$C$ is the unit circle and that the two sides  of $T$ containing $a=\infty, b$ and
 $a,c$, respectively, are rays parallel to the positive imaginary axis. It is then easy to see 
 that the only possibilities are $b=i$, $c=\rho$, in which case $T=T_0$, or $b=i$, $c=-\overline \rho$
 in which case $T$ is equal to the reflection image of $T_0$ in the imaginary axis, and hence M\"obius equivalent to $\overline {T_0}$.

If  $T\sub \bC$ has angles $0, \pi/2, 2\pi/3$ at its vertices $a,b,c\in \partial T$, respectively, then the quickest way to verify the statement is to reduce it to the first part. 
For this we consider the unique circular arc triangle $\widetilde T$ with 
the same vertices $a,b,c$ and angles $0, \pi/2, \pi/3$, respectively, so that $T$ is complementary to $\widetilde T$ in the sense discussed in the introduction. Then $\widetilde T\cup T$ forms  a 
lune as in Figure~\ref{fig:comple}. By what we have seen, we can find a M\"obius transformation $S$  that gives a  conformal map of $\widetilde T$ onto $T_0$ or $\overline{T_0}$. 
Then $S$ is a conformal map of $T$ onto $\overline {T_1}$ or $T_1$ and the statement follows.  \end{proof}  
The different cases in the previous lemma can easily be distinguished by taking boundary orientation into account. Let us assume  the circular arc triangle $T$ has angles $0, \pi/2, \pi/3$ at its vertices $a,b,c\in \partial T$, respectively. Then we orient $\partial T$ so that we traverse $a,b,c$ in this cyclic order if we run through $\partial T$ according to the orientation. Now suppose $T$ lies on the left of 
$\partial T$ with this orientation. Since M\"obius transformations are orientation-preserving, it then follows that $T$ cannot  be M\"obius equivalent to $\overline{T_0}$ and so it must be M\"obius
equivalent to $T_0$. If $T$ lies on the right of $\partial T$, then $T$ is M\"obius
equivalent to $\overline{T_0}$. A similar analysis applies to circular arc triangles that are M\"obius 
equivalent to $T_1$ and $\overline{T_1}$.  

The proof of our main result is now easy.
\begin{proof}[Proof of Theorem~\ref{main}] 
 We know that $\tau\mapsto J(\tau)$ sends $T_0$ conformally to  
 the circular arc triangle given by the closed lower halfplane. Here the vertex correspondence under the map is $\infty\mapsto \infty$, $i\mapsto 1$, $\rho\mapsto 0$. 
 On the other hand, we know that $p=H_1/H_2$, now considered as a function of 
 $J$, sends the lower $J$-halfplane conformally onto some circular arc triangle 
 $T$ with vertices  $a, b, c\in \bC$ where the  angles are  $0$, $\pi/2$, $2\pi/3$,  
 respectively. Here the vertex correspondence under $J\mapsto H_1/H_2$ is given by $0\mapsto c$, $1\mapsto b$,  $\infty\mapsto a$. 
 
 If we compose $\tau\mapsto J$ with $J\mapsto H_1/H_2$, then we obtain 
 the map $\tau \mapsto H_1(\tau)/H_2(\tau)=p(\tau)$. In particular, $p$ is a conformal map between the circular arc triangles $T_0$ and $T$ such that
 $p(\infty)=a$,  $p(i)=b$, $p(\rho)=c$.
 On the other hand, we know the values that $p$ obtains at these locations (see \eqref{p(i)},   \eqref{prho},  and \eqref{infty}). It follows that  $a=p(\infty)=\infty$,
 $b=p(i)=-i$,   $c=p(\rho)=\overline{\rho}$.
 
  If we orient $\partial T_0$ so that $T_0$ lies on the left of $\partial T_0$, then the vertices of $T_0$ are in cyclic order $ \infty, i, \rho$. Since the conformal map $p$ 
 of $T_0$ onto $T$ preserves orientation, $T$ lies on the left of $\partial T$ if $\partial T$ carries the induced orientation under $p$. This corresponds to  the cyclic order  $p(\infty)=\infty$,  $p(i)=-i$,  $p(\rho)=\overline {\rho}$ of the vertices of $T$. 
 
   Lemma~\ref{lem:triuniq} and the subsequent discussion after this lemma  now imply that 
   $T$ is M\"obius equivalent to $T_1$ (and not to $\overline{T_1}$). Since $T$ and $T_1$ have the same vertices, it follows that 
 $T=T_1$ and the statement follows.
 \end{proof}

\section{An alternative approach}\label{s:alt}
Our proof of Theorem~\ref{main} is in a sense the ``classical" proof. It may be interesting to point out a different and simple direct argument that avoids the theory of  hypergeometric ODEs. We require some preparation.

\begin{prop}\label{prop:arg}
Suppose $U\sub \bC$ is a closed Jordan region and $f\:\inte(U)\ra \C$ is a holomorphic function  that has a continuous extension (as a map into $\bC$) to the boundary $\partial U\sub \bC$. Suppose this extension maps $\partial U$ homeomorphically to the boundary $\partial V\sub \bC$ of a closed Jordan region $V\sub \bC$ with $ \inte(V)\sub \C$. 

If $\infty\in \partial V$ we make the following additional assumptions:
\begin{itemize} 

 \item[\textnormal{(i)}] if $z_0\in \partial U$ is the unique point with $f(z_0)=\infty$, then $f$ is locally injective on $U$ near $z_0$.
 
 \smallskip 
\item[\textnormal{(ii)}]  if the Jordan curve $\partial U$ is oriented so that $U$ lies on the left of $\partial U$ and if $\partial V$ carries the orientation induced by $f$, then $V$ lies on the left of $\partial V$.

\end{itemize}
Under these hypotheses, $f$ is a homeomorphism of   $U$ onto $V$ that is a biholomorphism  between $\inte(U)$ and $\inte(V)$.
\end{prop}
In our application of this statement, $\infty$ will be on the boundary of $U$ and $V$. This  is the reason why the formulation is somewhat technical and we cannot simply assume $U,V\sub \C$.   

\begin{proof}  Suppose first that $\infty\not \in \partial V$ (and hence $\infty\not \in V$). Then  the Argument Principle implies that on $\inte(U)$ the function $f$ attains each value in $\inte(V)$ once and no other values (see \cite[pp.~310--311]{Bur}  and 
 \cite[Exercise 9.17 (i)] {Bur} for this type of argument). The statement easily follows in this  case.   
 
 If $\infty \in  \partial V$, then the additional assumptions allow us to reduce to the previous case. For this we consider  auxiliary 
 Jordan regions $U'\sub U$ whose boundaries $\partial U'$ are obtained by replacing a small arc $\alpha\sub \partial U$ that contains $z_0$ in its interior with a small arc $\alpha'$ that has  the same endpoints, but whose interior is contained in $\inte(U)$ and so avoids $z_0$.  With suitable choices, condition (i) implies that $f$ is a homeomorphism of $\partial U'$ onto its image and $\infty\not \in f(\partial U')$. By the first part of the proof, there exists a Jordan region $V'\sub \C$ with $\partial V'=f(\partial U')$ such that $f$ is a ho\-meomorphism of $U'$ onto $V'$.  By choosing the arcs $\alpha$ and $\alpha'$ smaller and smaller, we  conclude that $f$ is a homeomorphism of $U$ onto its image. Condition (ii) then implies 
 $f(U)=V$ and  the statement is also true in this case.
 \end{proof}

We need  the following estimate for $E_2$ (see \cite[Lemma~2.3]{IJT}).

\begin{lem} \label{E2est}We have 
\begin{equation}
|E_2(\tau)-1|\le \frac{24|q|}{ (1-|q|)^3}
\end{equation}
for $\tau\in \H$. In particular,
\begin{equation}
E_2(\tau)\ne 0 \text{ for } \im(\tau)\ge \sqrt{3}/2.
\end{equation}
\end{lem}
\begin{proof} 
We have
$$ |E_2(\tau) -1| \le 
 24\sum_{n=1}^\infty \frac{n|q|^n}{1-|q|^n}
 \le 
 \frac{24}{1-|q|}\sum_{n=1}^\infty n|q|^n\\
 = \frac{24|q|}{(1-|q|)^3}.
$$
Since $|q|=e^{-2\pi\im(\tau)}$, the last expression is monotonically decreasing as a function of $\im(\tau)$. It follows that for $\im(\tau)\ge \sqrt{3}/2$ we have 
$$ |E_2(\tau) -1| \le \frac{24 e^{-2\pi \sqrt 3/2 } }
{(1-e^{-2\pi \sqrt 3/2})^3}=0.105\dots \, . 
$$ 
The statement follows.
\end{proof}

\begin{lem} The function $p$ is holomorphic near each point  $\tau \in T_0\cap \H$.
Moreover, $p'(\tau)=0$ for $\tau \in T_0\cap \H$ if and only if $\tau=\rho$.
\end{lem}
So the only critical point of $p$ on $T_0\cap \H$ is at $\tau=\rho$.
\begin{proof} The first part follows from \eqref{pE2}, because $E_2$ does not take the value $0$ on $T_0\cap \H$ by Lemma~\ref{E2est}. By \eqref{pE2} and \eqref{E2der} we also
have 
\begin{equation}\label{p'rep}
p'(\tau)=1+\frac{6iE_2'(\tau)}{\pi E_2^2(\tau)}=\frac{E_4(\tau)}{E_2^2(\tau)}.
\end{equation}
Now by \eqref{E4zeros} all zeros of $E_4$ are given by $S(\rho)$, $S\in \PSL_2(\Z)$.  No point $\tau\in T_0\setminus\{\rho\}$ is equivalent to  
$\rho$ under the action of  $\PSL_2(\Z)$ on $\H^*$. It follows that $\tau=\rho$ is the only zero of $E_4$ and hence of $p'$ on $T_0$.\end{proof}

\begin{proof}[Alternative proof of Theorem~\ref{main}]
We denote by 
\begin{align*}
A&\coloneqq \{it: t\ge 1\}\cup\{\infty\},\\
B&\coloneqq \{e^{it}: \pi/3\le t\le  \pi/2 \},\\
C&\coloneqq \{\tfrac12+it: t\ge \sqrt3/2 \} \cup\{\infty\}
\end{align*}
the three sides of $T_0$ and by  
\begin{align*}
A'&\coloneqq \{it: t\ge -1\}\cup\{\infty\},\\
B'&\coloneqq \{e^{it}: -\pi/2\le t\le  -\pi/3 \},\\
C'&\coloneqq \{\tfrac12+it: t\ge -\sqrt3/2 \} \cup\{\infty\}
\end{align*}
the three sides of $T_1$. 

It follows directly from the definition (see \eqref{E2}) that $E_2(\tau)\in \R$ when $\tau\in \H$ with $\re(\tau)=0$ or $\re(\tau)=1/2$. Here $E_2(\tau)>0$ if in addition $\im(\tau)$ is large enough.

As a consequence, $\re(p(\tau))=0$ if $\re(\tau)=0$ (see \eqref{pE2}). In particular,  $p$ sends $A\setminus\{\infty\} $ into the imaginary axis. Moreover, since $p'(\tau)\ne 0$ on $A$, we move strictly monotonically from $p(i)=-i$ to $p(\infty)=\infty$ as $\tau\in A$
travels from $i$ to $\infty$. Here $\im(p(\tau))$ must be strictly increasing (and not decreasing) with increasing $\im(\tau)$ for $\tau\in A$ as $\im(p(\tau))$ is positive 
for $\tau \in A$ with  large $\im(\tau)$. This implies that $p$ sends $A$ homeomorphically onto $A'$.  

Similarly, $\re(p(\tau))=1/2$ if $\re(\tau)=1/2$. Hence $p$ sends $C\setminus \{\infty\} $ into the line $
L\coloneqq \{\tau\in \C:\re(\tau)=1/2\}$. Since $p'(\tau)\ne 0$ for interior points of $C$, we 
move strictly monotonically from $p(\rho)=\overline \rho$ to $p(\infty)=\infty$ on the line $L$ as $\tau\in C$ moves from $\rho$ to $\infty$ on $C$. Again 
$\im(p(\tau))>0$ for $\tau\in C$  with  large $\im(\tau)$, and so we must  have a
 strict increase here. This implies that $p$ sends $C$ homeomorphically onto $C'$.  

Since $p$ sends $A$ into $A'$, by the Schwarz Reflection Principle we must have
$$p(-\overline \tau)=-\overline {p(\tau)}$$ 
for all $\tau\in \H$.
If $\tau\in B$, then $\tau$ lies on the unit circle and so $\overline \tau=1/\tau$. 
Then the transformation behavior of $p$  implies that
$$\overline {p(\tau)}= -p(-\overline \tau)=-p(-1/\tau)=1/p(\tau), $$
  and so $|p(\tau)|=1$. We see that $p$ sends $B$ into the unit circle. Since $p'$ does not vanish at any interior point of $B$, the point $p(\tau)$  moves strictly monotonically along the unit circle as 
  $\tau$ moves  from $i$ to $\rho$ along $B$. 
  
Now  $p'(i)\ne 0$, and so $p$ is a conformal and hence orientation-preserving  near $i$. We know that $p$ sends $A$ oriented in a positive direction from $i$ to $\infty$ into the arc $A'$ oriented from $-i$ to $\infty$. Since $B$ lies on  the right of the oriented arc $A$ locally near $i$, the initial part of the   image $p(B)$ must lie on the right of $A'$.
This is only possible if $p(\tau)$ moves  in counter-clockwise direction along the unit circle from $p(i)=-i$ to $p(\rho)=\overline {\rho}$ as $\tau$ moves from $i$ to $\rho$ along $B$.   In this  process the point $p(\tau)$ necessarily traces out the arc $B'$ at least once. In principle, it could trace out $B'$ more than once by wrapping around the unit circle repeatedly.
It this were the case, then  the moving points  $p(\tau)$ and $\tau$ would necessarily have to collide at least once somewhere in $B$; in other words, there would be a point $\tau\in B$ such that $p(\tau)=\tau$. But it immediately follows from 
\eqref{pE2} that $p$ cannot have fixed points in $\H$. 
We conclude that  $p$ maps  $B$ homeomorphically onto $B'$.

These considerations shows that  $p$ sends the three sides $A$, $B$, $C$ of $T_0$ homeomorphically onto the three sides $A'$, $B'$, $C'$ of $T_1$, respectively.
It follows that $p$ is a homeomorphism of $\partial T_0$ onto $\partial T_1$. Moreover, if we orient   $\partial T_0$ so that $T_0$ lies on the left of $\partial T_0$ and equip $\partial T_1$ with the orientation induced by $f$, then 
 $T_1$ lies on the left of $\partial T_1$.

In order to apply Proposition~\ref{prop:arg}, note that $p$ is holomorphic on the interior of $T_0$. To see that is is locally injective 
on $T_0$ near $\infty\in \partial T_0$ (which is the unique preimage
 of $\infty\in \partial T_1$ under $f$  in $T_0$), we first observe  that  
\eqref{pE2} and \eqref{E2} imply that 
\[
p'(\tau)= 1+O(q) 
\]
for $\tau\in \H$ with $\im(\tau)$ large. In particular, there exists 
$c_0>0$ such that $\re(p'(\tau))>0$ when $\tau\in \H$ and 
$\im(\tau)>c_0$. It follows that $p$ is injective on the convex region 
$\{ \tau\in \H: \im(\tau)>c_0\}$ (see \cite[Proposition 1.10]{Po}). In particular, $p$ is locally injective on $T_0$ near $\infty$.

 Proposition~\ref{prop:arg} now implies that $p$ is a conformal map of the circular arc triangle $T_0$ onto the circular arc triangle $T_1$. Here  the vertex correspondence is as  in the statement and the claim follows.      
 \end{proof}
 
 We want to point out a consequence of our considerations.
 \begin{cor}\label{cor:critp}
 The critical points of $p$ on $\H$ are precisely the points that are equivalent to $\rho$ under the action of $\PSL_2(\Z)$. 
 \end{cor}
  
 \begin{proof}  It follows from Corollary~\ref{TT} and the Schwarz Reflection Principle  that  if $T$ is a triangle from the tessellation $\mathcal{T}$ (as defined in the introduction), then $p$ is a conformal map near each point $\tau\in T\cap \H$ distinct from the vertex $\tau_T$  of $T$ where the angle is $\pi/3$. There the angle is doubled by $p$ and so $\tau_T$ is a critical point (of first order).  
 Therefore, $p$ has precisely the critical points $\tau_T$, $T\in \mathcal{T}$. These are precisely the points that are equivalent to $\rho$ under the action of $\PSL_2(\Z)$. \end{proof}
 
 As the proof shows, all critical points of $p$ are of first order, meaning that the local degree of the map is $2$ at these points. Corollary~\ref{cor:critp} can also be derived from  \eqref{p'rep} and  \eqref{E4zeros}.

  \section{Remarks} \label{s:rems}
 1. Our route to derive Proposition~\ref{prop:hypp} is essentially due to Fricke--Klein. In contrast to our approach,  
 Fricke--Klein usually prefer to work with homogeneous modular forms.
 If $\widetilde f(\om_1, \om_2)$ is  such a form, then they use  an 
 associated form 
 $$ D_\eta \widetilde f= \eta_1\frac{\partial \widetilde f}{\partial \om_1} +
  \eta_2\frac{\partial  \widetilde f}{\partial \om_2}.$$ 
If $\widetilde f$ is a homogeneous modular form of degree $-k$, then  it is straightforward to see that 
 $\widetilde g=D_\eta \widetilde f$ is a homogeneous modular form of degree $-k-2$. The 
 differentiation process $D_\eta$ occurs implicitly in Fricke--Klein \cite{FK} and was systematically used by Fricke in \cite{Fr} (see also 
  \cite{FS}).
 
 Let us consider the associated inhomogeneous forms $f$ and $g$ so that 
 $$\widetilde f(\om_1, \om_2)=\biggl(\frac{2\pi}{\om_2}\biggr)^k f(\om_1/\om_2)
 \text{ and }  \widetilde g(\om_1, \om_2)=\biggl(\frac{2\pi}{\om_2}\biggr)^{k+2}g(\om_1/\om_2).$$
 Then 
\begin{align*} \widetilde g &=D_\eta \widetilde f= \eta_1\frac{\partial \widetilde f}{\partial \om_1} +
  \eta_2\frac{\partial  \widetilde f}{\partial \om_2}\\ &=
(2\pi)^k (\eta_1\om_2^{-k-1}- \eta_2\om_1\om_2^{-k-2} )f'(\tau)-k(2\pi)^k \eta_2\om_2^{-k-1}f(\tau)\\
 &=-(2\pi)^{k+1} i \om_2^{-k-2} f'(\tau)-\frac{k}{12} (2\pi)^{k+2} \om_2^{-k-2} E_2(\tau) f(\tau).
 \end{align*}
 So passing to the inhomogeneous version $g$ of $\widetilde g$, we have 
 $$ g(\tau)=\frac{1}{(2\pi)^{k+2}}\widetilde g(\tau, 1)=\frac{1}{2\pi i} f'(\tau)-\frac{k}{12} E_2(\tau) f(\tau), $$ 
 and so 
 $$ g=Df-\frac{k}{12} E_2f= \vartheta_kf. $$ 
This shows that  the Serre  derivative $\vartheta_kf$ of a modular form $f$ of  weight
$k$ corresponds to the homogenous form $D_\eta\widetilde f$.  So Serre derivative 
and the differentiation process $D_\eta$ are essentially the same, but only differ whether one considers inhomogeneous or homogeneous modular forms.  

2. The meaning of the functions   in \eqref{Om} and \eqref{H} becomes more transparent if one writes them using  the homogeneous function 
$$\widetilde \Delta (\om_1, \om_2)=\biggl(\frac{2\pi}{\om_2}\biggr)^{12}\Delta(\om_1/\om_2)$$ and the associated root 
$$\widetilde \Delta^{1/12} (\om_1, \om_2)=
\biggl(\frac{2\pi}{\om_2}\biggr)\Delta^{1/12}(\tau).$$
Then $$\Om_1=\frac{\om_1}{2\pi} \widetilde \Delta^{1/12}(\om_1,\om_2), \quad 
\Om_2=\frac{\om_2}{2\pi}\widetilde \Delta^{1/12}(\om_1,\om_2). $$  So $\Om_1$ and $\Om_2$ are essentially just  the periods $\om_1$ and $\om_2$ renormalized to make them homogeneous of 
degree $0$. Similarly, using \eqref{eta1} and \eqref{eta2}, one can see that 
$$H_1=\frac{6\eta_1 }{\pi} \widetilde\Delta^{-1/12} (\om_1,\om_2), \quad H_2=\frac{6 \eta_2 }{\pi} \widetilde \Delta^{-1/12}(\om_1,\om_2).$$
Again here $H_1$ and $H_2$ are homogeneous of degree $0$ as functions of 
the pair $(\om_1, \om_2)$ and hence  functions of $\tau=\om_1/\om_2$ alone; so 
the transition  from $(\eta_1, \eta_2)$ to $(H_1, H_2)$ is also a normalization procedure.

3. We can use Lemma~\ref{OHderiv} also to derive a hypergeometric differential equation  with a pair of fundamental solutions given by $\Om_1$ and $\Om_2$. Indeed, if we use
  notation as in the proof of Proposition~\ref{prop:hypp}, then we have
\begin{align*}
\frac{d^2\Om}{dJ^2}& =   \frac{d}{dJ}\biggl( -\frac{1}{24\sqrt3}  J^{-2/3} (J-1)^{-1/2}  H\biggr)\\
&=\biggl(-\frac {2}{3J} -\frac {1}{2(J-1)}\biggr) \frac{d\Om}{dJ}-
\frac{1}{24\sqrt3}  J^{-2/3} (J-1)^{-1/2}  \frac{dH}{dJ}\\
&=-\frac{7J-4}{6J(J-1)}\frac{d\Om}{dJ}-\frac{1}{144J(J-1)}\Om. 
\end{align*} 
This is a hypergeometric  differential equation  with parameters $\alpha=\beta=1/12$ and $\ga=2/3$. We conclude that $\tau =\Om_1/\Om_2$ as a function of $J$  maps the closed upper and lower $J$-halfplanes onto  circular arc triangles with angles $\la\pi$,  $\mu\pi$, $\nu\pi$, where
$$ \la=1-\ga=1/3,\quad \mu=\ga-\alpha-\beta=1/2,\quad \nu=\alpha-\beta=0.$$ 
Indeed, this is in agreement with the fact, pointed out earlier, that $\tau\in T_0 \mapsto J(\tau)$ is a conformal map of $T_0$ onto the closed lower
 halfplane.

4. The {\em Schwarzian derivative} of a meromorphic function $f(z)$ depending on a complex variable $z$ is defined as 
\begin{align*}
\{f,z\}&\coloneqq \frac{d^2}{dz^2}\log f'(z)-\frac 12 \biggl(\frac{d}{dz}\log f'(z)\biggr)^2\\
&=  \frac{2f'(z)f'''(z)-3f''(z)^2}{2f'(z)^2}.
\end{align*}
It is well known (see \cite[p.~253]{Bie}) that the ratio $f=w_1/w_2 $ of two fundamental solutions of the hypergeometric differential equation has a Schwarzian derivative given by
$$ \{f, z\} =\frac{1-\la^2}{2z^2} + \frac{1-\mu^2}{2(1-z)^2}+ 
\frac{1-\lambda^2-\mu^2+\nu^2}{2z(1-z)},
$$ where
$$ \la=1-\ga,\quad \mu=\ga-\alpha-\beta,\quad \nu=\alpha-\beta.$$ 

For our function $p=H_1/H_2$ we have $\la=2/3$, $\mu=1/2$, $\nu=0$, and so 
$$\{p,J\}=\frac{5}{18J^2} + \frac{3}{8(1-J)^2}+ 
\frac{11}{72J(1-J)}. $$

Similarly,
$$\{\tau,J\}= \frac{4}{9J^2} + \frac{3}{8(1-J)^2}+ 
\frac{23}{72J(1-J)}.
$$ 

 The Schwarzian derivative of a function $f$ is invariant under post-com\-po\-si\-tion with a  M\"obius transformations $S$, namely,
$$ \{S\circ f, z\} =\{f, z\}.$$ 
Moreover, we have the following chain rule for the Schwarzian
derivative $$\{f\circ g, z\}=\{f,g\}\biggl(\frac{dg}{dz}\biggr)^2+\{g,z\}.$$

If $S\in \PSL_2(\Z)$ is arbitrary, then by $\PSL_2(\Z)$-equivariance of our function $p$ we have $S\circ p=p\circ  S$. 
It follows that 
$$f(\tau) \coloneqq \{p,\tau\} =\{S\circ p, \tau\} =\{p\circ S,\tau\}= f(S(\tau))S'(\tau)^2. $$
This implies that $f$ has the same transformation behavior as a modular form  of weight $4$.  
 
 This can be seen more explicitly as follows:
\begin{align*}
\{p,\tau\} &= \{p, J\} J'(\tau)^2+\{J,\tau\} = (\{p, J\} -\{\tau, J\})J'(\tau)^2 \\
&= -\frac{J'^2}{6}\biggl(\frac{1}{J^2}  +
\frac{1}{J(1-J)}\biggr)=\frac{J'^2}{6J^2(J-1)}\\
&=-8\pi^2\Delta^{1/3}/J^{2/3}=-1152\pi^2\Delta/E_4^2.
\end{align*} 
This last expression for $\{p,\tau\}$ was also recorded 
in \cite[Proposition 6.2]{SS12}.  Since $\Delta$ has weight $12$ and $E_4$ has weight $4$, it clearly  transforms as a modular form of weight $4$. Note that  $\{p,\tau\}$ is  not a modular form
according to our definition as this function  has poles, namely exactly at the points equivalent to $\rho$ under $\PSL_2(\Z)$.  
These points are precisely the critical points of $p$
as we know from Corollary~\ref{cor:critp}.

\end{document}